\documentclass[12pt, reqno]{amsart}
\usepackage{amsthm}
\usepackage{amsfonts}
\usepackage{amscd}
\usepackage{amssymb}
\usepackage{mathrsfs}

\numberwithin{equation}{section}

\theoremstyle{plain}

\newtheorem*{theorem*}{Theorem}

\newtheorem{theorem}{Theorem}[section]
\newtheorem{proposition}[theorem]{Proposition}

\theoremstyle{definition}

\newtheorem{remark}[theorem]{Remark}

\newcommand{\R}{{\mathbb R}}
\newcommand{\N}{{\mathbb N}}

\newcommand{\C}{{\mathbb C}}

\renewcommand{\S}{{\mathcal S}}

\newcommand{\wpsmooth}{{C_{w,p}^\infty(\Omega)}}

\newcommand{\Wa}{W_\alpha}
\newcommand{\ltwo}{L_2(\R,dx)}
\newcommand{\alp}{{L_p(\R, \Wa^p\,dx)}}
\newcommand{\altwo}{{L_2(\R, \Wa^2\,dx)}}
\newcommand{\twoltwo}{L_2(\R, e^{-2x^2}dx)}
\newcommand{\anorm}{{\altwo}}
\newcommand{\asmooth}{{C_\alpha^\infty(\R)}}

\newcommand{\ajseminorm}{q_{\alpha, j}}
\newcommand{\afourier}{\mathcal F_\alpha}
\newcommand{\iso}{{\Phi}}
\newcommand{\aiso}{\iso_\alpha}

\newcommand{\wpaseminorm}{{q_{w,p,\alpha}}}

\newcommand{\seq}[1]{(#1)}

\newcommand{\pdeglegn}{{\Pi_n}}

\newcommand{\raspace}{H_{\textup{ra}}}

\newcommand{\araspace}{L_2(\R,\Wa^2\,dx)_{\textup{ra}}}
\newcommand{\tworaspace}{L_2(\R,e^{-2x^2}dx)_{\textup{ra}}}

\newcommand{\fourierhermite}{{\mathcal F}_h}

\begin{document}

\title[Rapid polynomial approximation in ${L_2}$-spaces with Freud weights]{Rapid polynomial approximation in $\boldsymbol{L_2}$-spaces with Freud weights on the real line}
\author{Rui Xie}
\address{Rui~Xie,
Department of Mathematics,
Harbin Institute of Technology
Harbin, 150001, P.R.\ China}
\email{xierui303030@126.com}
\thanks{The first author was supported by a research grant from China Scholarship Council}
\author{Marcel de Jeu}
\address{Marcel~de~Jeu, Mathematical Institute,
Leiden University,
P.O.\ Box 9512,
2300 RA Leiden,
The Netherlands}
\email{mdejeu@math.leidenuniv.nl}
\date{}

\begin{abstract}
The weights $\Wa(x)=\textup{exp}(-|x|^{\alpha})$ $(\alpha>1)$ form a subclass of Freud weights on the real line. Primarily from a functional analytic angle, we investigate the subspace of $\altwo$ consisting of those elements that can be rapidly approximated by polynomials. This subspace has a natural Fr\'echet topology, in which it is isomorphic to the space of rapidly decreasing sequences. We show that it consists of smooth functions and obtain concrete results on its topology. For $\alpha=2$, there is a complete and elementary description of this topological vector space in terms of the Schwartz functions.
\end{abstract}
\maketitle
 \section{Preliminary results and overview}\label{sec:intro}

In this paper we are concerned with rapid polynomial approximation of functions in a weighted $L_2$-space on the real line, where the weight is defined in terms of a particular class of Freud weights. Using classical approximation results, we approach the situation primarily from a functional analytic point of view, as we will now explain.

\subsection{Rapidly approximable functions}
To motivate our work, we start by considering the following more general situation. Let $\Omega\subset\R^d$ be a nonempty measurable subset, and let $\mu$ be a non-negative measure on $\Omega$. Let $\Pi$ denote the polynomials on $\R^d$, and let $\Pi_n$ $(n\geq 0)$ denote the polynomials of degree at most $n$. Fix $1\leq p\leq \infty$, and assume that $\Pi$ (or rather the set of restrictions of the polynomials to $\Omega$) is contained in $L_p(\Omega,\mu)$; for $1\leq p<\infty$ this implies that $\mu$ must be finite. Assume further that $\Pi$ is dense in $L_p(\Omega,\mu)$. Equivalently, assume that
\begin{equation}\label{eq:decreasestozero}
d(f,\Pi_n)_{L_p(\Omega,\mu)}\downarrow 0,
\end{equation}
for all $f\in L_p(\Omega,\mu)$. Here $d(f,\Pi_n)_{L_p(\Omega,\mu)}$ denotes the distance $\inf_{P\in\Pi_n}||f-P ||_{L_p(\Omega,\mu)}$ of $f$ to the closed linear subspace $\Pi_n$ of $L_p(\Omega,\mu)$. If $w:\N_0\to\R_{\geq 0}$ is increasing, and $f\in L_p(\Omega,\mu)$ is such that
\begin{equation*}
\sup_{n\geq 0} w(n)d(f,\Pi_n)_{L_p(\Omega,\mu)}<\infty,
\end{equation*}
then -- especially if $w$ increases to $\infty$ --  this gives more information than \eqref{eq:decreasestozero} about the gain in accuracy of approximating $f$ by elements of $\Pi_n$ as $n$ increases. We will say that $f\in L_p(\Omega,\mu)$ is \emph{rapidly approximable in $L_p(\Omega,\mu)$} (by polynomials) if
\begin{equation*}
\sup_{n\geq 0} n^k d(f,\Pi_n)_{L_p(\Omega,\mu)}<\infty,
\end{equation*}
for all $k\geq 0$. Here, as elsewhere, $0^0$ is to be read as 1. Equivalently, if we put $\Pi_{-1}=\{0\}$, we could have required
\begin{equation*}
\sup_{n\geq 0} n^k d(f,\Pi_{n-1})_{L_p(\Omega,\mu)}<\infty,
\end{equation*}
for all $k\geq 0$, which is technically more convenient (cf.\ the proof of Theorem~\ref{the:HilbertFrechet}). The set $L_p(\Omega,\mu)_{\textup{ra}}$ of rapidly approximable elements of $L_p(\Omega,\mu)$ is clearly an abstract linear subspace of $L_p(\Omega,\mu)$. If $p=2$, which is our main interest, Theorem~\ref{the:HilbertFrechet} implies that $L_2(\Omega,\mu)_{\textup{ra}}$ is isomorphic as an abstract vector space to the space $(s)$ of rapidly decreasing sequences. The map realizing this isomorphism sends $f\in L_2(\Omega,\mu)$ to its sequence $\seq{a_{P_n}(f)}$ of Fourier coefficients with respect to a system $\seq{P_n}$ of orthonormal polynomials of nondecreasing degree for $\mu$.  Moreover, if we supply $L_2(\Omega,\mu)$ with the locally convex topology induced by the family $\{q_k : k=0,1,2,\ldots\}$ of seminorms on $L_2(\Omega,\mu)$, defined by
\begin{equation*}
q_k(f)=\sup_{n\geq 0} n^k d(f,\Pi_{n-1})_{L_2(\Omega,\mu)}\quad(f\in L_2(\Omega, \mu)_{\textup{ra}}),
\end{equation*}
then $L_2(\Omega, \mu)_{\textup{ra}}$ is a Fr\'echet space, the inclusion $L_2(\Omega, \mu)_{\textup{ra}}\subset L_2(\Omega,\mu)$ is continuous, and the isomorphism of  $L_2(\Omega, \mu)_{\textup{ra}}$ with $(s)$ as above is a topological isomorphism of Fr\'echet spaces.

\subsection{Concrete models for the rapidly approximable functions}\label{subsec:concretemodels}
For concrete realizations to $\Omega$ and $\mu$, the natural questions are to describe the vector space $L_2(\Omega,\mu)_{\textup{ra}}$ as concretely as possible, and also to determine its Fr\'echet topology as explicitly as possible. Though not motivated from our viewpoint, this program has in fact been completed in at least two general papers, both for bounded sets, which we now discuss.

First of all, Zerner \cite{Zerner} (see \cite{Pavec} for the proofs) showed the following. Let $\Omega$ be an open bounded subset of $\R^d$ with Lipschitz boundary, let $D(\bar{\Omega})$ be the Fr\'{e}chet space of functions having derivatives of all order on the closure $\bar{\Omega}$ of $\Omega$. Suppose that $m\in L^{1}(\Omega,dx)$, and that there exists $\delta>0$ such that $m(x)\geq \delta$, for all $x\in \Omega$. Choose a system $\seq{P_n}$ of orthonormal polynomials of nondecreasing degree for $\mu$. Then the map sending $f\in D(\bar{\Omega})$ to its sequence $\seq{a_{P_n}}$ of Fourier coefficients with respect to $\seq{P_n}$ establishes a topological isomorphism between $D(\bar{\Omega})$ and the space of rapidly decreasing sequences. Thus, interpreted in our framework, $L_2(\Omega, m(x)\,dx)_{\textup{ra}}=D(\bar{\Omega})$ as topological vector spaces.

The condition that $m$ is strictly bounded away from zero excludes, for example, the weight for the Jacobi polynomials for various values of the parameters. Fortunately, there is a stronger result that covers these weights as well, and in fact also covers suitable measures for which there need not be a density. This is due to Zeriahi \cite{Zeriahi}, which is our second paper to be discussed.

Suppose $\Omega$ is a nonempty compact subset of $\R^d$ with Lipschitz boundary. Let $\mu$ be a Borel measure on $\Omega$. Suppose that there exist $C>0,\gamma>0$ and $0<t_0<1$ such that, for every $x$ in $\Omega$, $\mu(\Omega\cap B(x,t))\geq C t^\gamma$ for all $0<t<t_{0}$, where $B(x,t)$ is the ball with center $x$ and radius $t$. Let $C^\infty(\Omega)$ be the space of all functions on $\Omega$ that can be extended to a smooth function on $\R^d$, in the Fr\'echet topology described as on \cite[p.~689]{Zeriahi}. Choose a system $\seq{P_n}$ of orthonormal polynomials of nondecreasing degree for $\mu$. Then it follows from \cite[Th\'eor\`eme~3.1]{Zeriahi} that the map sending $f\in C^\infty(\Omega)$ to its sequence $\seq{a_{P_n}}$ of Fourier coefficients with respect to $\seq{P_n}$ establishes a topological isomorphism between the Fr\'echet space $C^\infty(\Omega)$ and the space of rapidly decreasing sequences. Thus, interpreted in our framework again, $L_2(\Omega,\mu)_{\textup{ra}}= C^\infty(\Omega)$ as topological vector spaces. Let us note that \cite{Zeriahi} contains much more material than just cited, and also that -- with the measure as indicated -- the topological isomorphism statement holds under more lenient conditions on the geometry of $\Omega$ than having Lipschitz boundary. It is sufficient for $\Omega$ to be a compact uniformly polynomial cuspidal set; see \cite{PawPle} and \cite[p.~684]{Zeriahi}.

Returning to the general context again, we note that a typical way of guaranteeing that a candidate function $f$ is in $L_2(\Omega,\mu)_{\textup{ra}}$, is to show that it falls within the scope of a suitable Jackson-type inequality. There is an extensive literature on such inequalities. As an example, taken from \cite[p.~815]{Ple}, we let $\Omega$ be a fat (i.e., $E=\overline{\textup{int}\, E}$) compact subset of $\R^d$. Then it is said in \cite{Ple} that $\Omega$ admits a Jackson inequality if, for each $k=0,1,2,\cdots$, there exist $C_k>0$ and an integer $m_{k}\geq 0$ such that, for all $f\in C_{\textup{int}}^{\infty}(E)$ and all $n>k$,
\begin{equation}\label{eq:fatJackson}
n^{k}d(f,\Pi_n)_{L_\infty(\Omega,dx)}\leq C_{k}\sum_{|\alpha|\leq m_{k}}|| D^{\alpha}f(x)||_{L_\infty(\Omega,dx)}.
\end{equation}
Here $C_{\textup{int}}^{\infty}(E)$ denotes the space of smooth functions on $\textup{int}\, E$ that can be continuously extended to $E$, together with all their partial derivatives. If $\Omega$ admits a Jackson inequality, and $f\in C_{\textup{int}}^\infty(E)$, then it is immediate that, for $n>k$,
\begin{equation*}
n^{k}d(f,\Pi_n)_{L_2(\Omega,dx)}\leq C_{k}\left(\sum_{|\alpha|\leq m_{k}}|| D^{\alpha}f(x)||_{L_\infty(\Omega,dx)}\right)\left(\int_\Omega 1\,dx\right)^{1/2},
\end{equation*}
so certainly
\begin{equation*}
\sup_{n\geq 0} n^{k}d(f,\Pi_n)_{L_2(\Omega,dx)} < \infty
\end{equation*}
for all $k=0,1,2,\ldots$. Hence $C_{\textup{int}}^{\infty}(E)\subset L_2(\Omega,dx)_{\textup{ra}}$ in that case. Note, however, that any constant depending on $k$ and $f$ in the right hand side of \eqref{eq:fatJackson} would have sufficed to reach this conclusion. It is not necessary to have an upper bound in \eqref{eq:fatJackson} depending on $k$ and $f$ as it does.

Conversely, if one wants to deduce that $L_2(\Omega,dx)_{\textup{ra}}$ is contained in a candidate space, this can sometimes be done using a Markov-type inequality, on which the literature is likewise extensive. As an example, in \cite[p.~450]{Ples} it is said that a compact subset $\Omega$ of $\R^d$ is Markov, if there exist constants $M>0$ and $r>0$ such that, for all $P\in\Pi$,
\begin{equation*}
\|\textup{grad}\,P\|_{L_\infty(\Omega,dx)}\leq M (\textup{deg}\,P)^r \| P \|_{L_\infty(\Omega,dx)}.
\end{equation*}
By iteration this enables one to control the derivatives of $P\in\Pi_n$ in terms of $P$ at the cost of a factor that is polynomial in $n$. One then transfers such an inequality to the norm in $L_2(\Omega,\mu)$ and obtains that, if $\seq{a_{P_n}}$ is rapidly decreasing, the series $\sum_{P_n} a_{P_n} D^\alpha P_n$ -- where $\seq{P_n}$ is a system of orthonormal polynomials for $L_2(\Omega,\mu)$ with nondecreasing degree -- is convergent in a suitable topology. This will then typically allow one to conclude that elements of $L_2(\Omega,dx)_{\textup{ra}}$ are smooth in the sense as applicable in the particular situation at hand.

If all is well, one has two opposite inclusions and $L_2(\Omega,dx)_{\textup{ra}}$ has been determined as a vector space. If there is a natural Fr\'echet topology on  $L_2(\Omega,dx)_{\textup{ra}}$, then the Closed Graph Theorem and Open Mapping Theorem can be convenient to show that the isomorphism must necessarily be topological, cf.~\cite{Zerner} or \cite[p.~693]{Zeriahi}.

\subsection{Rapid approximation on the real line with Freud weights}
The discussion above has been mainly for bounded subsets of $\R^d$, but clearly one can ask the same question to describe $L_2(\Omega,\mu)_{\textup{ra}}$ as a topological vector space for a concrete unbounded $\Omega$ and (bounded) $\mu$ such that $\Pi\subset L_2(\Omega,\mu)$. Much less is known here. To our knowledge the present paper may be the first to consider this question for unbounded $\Omega$, and Theorem~\ref{the:alphaistwo} below is the only case we are aware of where this question has been answered in full.

The problem is that weighted approximation on unbounded subsets of $\R^d$ is much harder than on unbounded subsets. In particular, the Jackson-type and Markov-type inequalities, that lie at the basis of the topological isomorphisms in \cite{Zerner, Zeriahi} as discussed above, are far less well developed.

In one dimension, however, there are some results available when $\Omega=\R$ and $\mu=W(x)\,dx$ where $W$ is a so-called Freud weight. Following the modern definition \cite[Definition~3.3]{Lu}, $W:\R\to\R_{>0}$ is a Freud weight if it is of the form $W=\textup{exp}(-Q(x))$, where $Q:\R\rightarrow \R$ is even, $Q'$ exists and is positive on $(0,\infty)$, $xQ'(x)$ is strictly increasing on $(0,\infty)$, with right limit $0$ at $0$, and such that, for some $\lambda,A,B>1$, and $C>0$,
\begin{equation*}
A\leq \frac{Q'(\lambda x)}{Q'(x)}\leq B\;\textup{for}\;x\geq C.
\end{equation*}
Such weights, first introduced by Freud, have received considerable attention; see, e.g., \cite{DitTot,Freud1, Freud2, Freud3, LeLu1, LeLu2,Mha}. Clearly the weights $\Wa(x)=\textup{exp}(-|x|^{\alpha})$ $(\alpha > 1)$ are Freud weights.

Our aim is to describe the topological vector space $\araspace$ for $\alpha>1$ as much as possible.\footnote{As far as the choice of the weight is concerned, the space $L_2(\R,W_\alpha\,dx)_{\textup{ra}}$ would perhaps be a more natural choice. However, these spaces are isometrically isomorphic to $\araspace$ via a dilation (which leaves each $\Pi_n$ invariant!), and the squared version is more in concordance with the notation in various approximation results in the literature.}

First of all, let us note that $\Pi\subset L_p(\R,W_\alpha^p\,dx)$ $(\alpha>0, 1\leq p<\infty)$ and that, as a special case of the general polynomial density result in $L_p$-spaces for quasi-analytic weights \cite[Corollary~6.34]{Jeu}, which in itself is a consequence of more general considerations applicable in a variety of topological function spaces, this subspace is dense if $\alpha\geq 1$ and $1\leq p<\infty$. Aside, for the sake of completeness we mention that uniform polynomial approximation with weight $\Wa$ for $\alpha>0$ has also been considered (a special case of Bernstein's original problem) and that it is known that the polynomials are then dense if and only if $\alpha\geq 1$; see \cite[p.~254]{Lu1} for details and references.

Of course $\Pi\subset\araspace$, but can we see other elements? With the discussion above in mind, the first thing to look for is a Jackson-type inequality. Indeed there is one, as given by the following theorem, which is a consequence of iterating \cite[Corollary~3.2]{Lu} as on \cite[p.~12-13]{Lu} and the formula for the Mhaskar-Rakhmanov-Saff number figuring therein, cf.~\cite[p.~11]{Lu}.

\begin{theorem}\label{the:aJackson}
Let $1\leq p<\infty$, $r\geq 1$ and $\alpha>1$. Then there is a constant $C$ with the following property:

If $f\in C^{r-1}(\R)$, $f^{(j)}$ is absolutely continuous for $j=0,1,\ldots, r-1$, and $\|f^{(r)}\|_{\alp}<\infty$, then, for all $n\geq 0$,
\begin{equation}\label{eq:aJackson}
 n^{r\left(\frac{1} {\alpha}-1\right)} d(f,\Pi_n)_{\alp}\leq C\| f^{(r)}\|_{\alp}.
\end{equation}
\end{theorem}

\begin{remark}
Lubinsky has shown \cite[p.~255]{Lu1} that there is no Jackson inequality as in \eqref{eq:aJackson} for $\alpha=1$. The amount of work needed to establish this and related results like Theorem~\ref{the:aJackson} and the Markov inequality in Theorem~\ref{the:aMarkov} is formidable.
\end{remark}

The following is immediate from Theorem~\ref{the:aJackson}.

\begin{proposition}\label{prop:subspace}
 Let $\alpha>1$. If $f\in C^\infty(\R)$, $f^{(r)}$ is absolutely continuous and $\|f^{(r)}\|_{\alp}<\infty$ for all $r\geq 0$, then $f\in\araspace$.
\end{proposition}

The set of all $f$ as in Proposition~\ref{prop:subspace} is a vector space, and together with $\Pi$ it spans a subspace $L_\alpha$ of $\araspace$. Note, however, that an absolutely continuous function on $\R$ cannot grow faster than linearly, so that all elements of $L_\alpha$ are of at most polynomial growth. Since $\altwo$ contains functions of superexponential growth, $L_\alpha$ seems suspiciously small to be a candidate for $\araspace$. For $\alpha=2$ we know in fact from Theorem~\ref{the:alphaistwo} that $\tworaspace=\{ge^{x^2} : g\in\S(\R)\}$, where $\S(\R)$ is the usual Schwartz space of rapidly decreasing functions. This space contains functions of superexponential growth and the feeling arises that this could be the general phenomenon. At the moment, however, for $\alpha>1, \alpha\neq 2$, we can only conclude that $L_\alpha\subset\araspace$. To improve this one would need an inequality as \eqref{eq:aJackson}, valid for more functions and allowing a constant on the right hand side that need not depend on $f$ and $r$ as in \eqref{eq:aJackson}. We will see in Proposition~\ref{prop:equivalences} how such an equality is intimately connected with a possible candidate for $\araspace$.

We turn to the other part of the program as sketched in Section~\ref{subsec:concretemodels}, namely determining a priori regularity properties of elements of $\araspace$. Here we can do more. With Section~\ref{subsec:concretemodels} in mind, the following Markov inequality \cite[equation~(7.3)]{Lu} for the weights $\Wa(x)$ is expected to be useful.
\begin{theorem}\label{the:aMarkov}
Let $p\in(0,\infty]$ and $\alpha>1$. Then there exists a constant $C_{\alpha,p}$\ such that, for all $P\in\pdeglegn$,
\begin{equation}\label{eq:Markov}
\|P'\|_{\alp}\leq C_{\alpha,p} n^{1-\frac{1}{\alpha}}\|P\|_{\alp}.
\end{equation}
\end{theorem}

As it turns out, it is indeed possible to put this to good use, and as a by-product one also obtains some first information on the topology of $\araspace$ in the process. However, doing so takes some effort, and this is the main body of work in this paper. It may be a reflection of the intrinsic difficulty of approximation on unbounded sets that even then our results, although non-trivial, are not complete (with the exception of $\alpha=2$), showing that further research is still necessary to understand the rapidly approximable functions in this case of an unbounded underlying set.

For the convenience of the reader, we collect our main results in the following theorem that summarizes the results of Section~\ref{sec:regularityandtopology} and \ref{sec:alphaistwo}. In its formulation, $\seq{P_{\alpha,n}}$ is the real-valued polynomial orthonormal basis of $\altwo$, where $\textup{deg}\,p=n$ $(n\geq 0)$. The Fourier coefficients of $f\in \altwo$ with respect to this basis are given by
\begin{equation*}
a_{\alpha,n}(f)=\int_{\R}f(x)P_{\alpha,n}(x)\Wa^{2}(x)\,dx\quad(n\geq 0).
\end{equation*}

\begin{theorem}[Main results]\label{the:mainresults}
 Let $\alpha>1$ and let $\araspace$ be the subspace of elements of $\altwo$ that can be rapidly approximated by polynomials. Let
\begin{equation*}
\asmooth =\{f\in C^{\infty}(\R):f^{(j)}\in \altwo ,\, j=0,1,2,\cdots\}.
\end{equation*}
Supply $\asmooth$ with the locally convex topology induced by the family of seminorms $\{\ajseminorm : j=0,1,2,\ldots\}$, defined, for $j=0,1,2,\cdots$,  by
\begin{equation*}
\ajseminorm(f)=\|f^{(j)}\|_\anorm\quad(f\in\asmooth).
\end{equation*}
This space $\asmooth$ is a Fr\'echet space. If $f_n\to f$ in $\asmooth$, then $f_n^{(j)}\to f^{(j)}$ uniformly on  compact subsets of $\R$, for all $j\geq 0$.

The non-trivial inclusion $\araspace\subset\asmooth$ holds, and the inclusion map is continuous. In particular, if $f_n\to f$ in $\araspace$, then $f_n^{(j)}\to f^{(j)}$ uniformly on  compact subsets of $\R$, for all $j\geq 0$.

If $f\in\araspace$, and $\seq{a_{\alpha,n}(f)}$ is its sequence of Fourier coefficients with respect to the orthonormal basis $\seq{P_{\alpha,n}}$ of $\altwo$, then $\sum_{n=0}^\infty a_{\alpha, n}(f) P_{\alpha,n}$ converges to $f$ in $\araspace$. Consequently, $\sum_{n=0}^\infty a_{\alpha, n}(f) P_{\alpha,n}^{(j)}$ converges uniformly to $f^{(j)}$ on compact subsets of $\R$, for all $j=0,1,2,\ldots$.

If $\alpha=2$, then
\[
\tworaspace=\{ge^{x^2} : g\in\S(\R)\},
\]
and the resulting bijection between $\araspace$ and $\S(\R)$ is a topological isomorphism.
\end{theorem}

\begin{remark}\label{rem:ditzianpaper}
It should be mentioned here that \cite{Dit} sheds some additional light on $\araspace$. In that paper, the relation between weighted integrability of a function and weighted summability of its Fourier coefficients is investigated. If $\alpha>1$ and $f\in\araspace$, it is a consequence of \cite[Theorem~2.3]{Dit} that
\begin{equation*}
\int_\R |f(x)|^q e^{-q|x|^\alpha} \,dx<\infty\quad(2\leq q <\infty),
\end{equation*}
and that there exists a constant $C$ such that
\begin{equation}\label{eq:bounded}
|f(x)|\leq C e^{|x|^\alpha}.
\end{equation}
\end{remark}

\medskip
This paper is organized as follows.

In Section~\ref{sec:rapidapproximation} we define the notion of rapidly approximable elements of a separable Hilbert space. It carries a natural topology and it follows from a more general result that it is then a Fr\'echet space. With this topology it is shown to be topologically isomorphic to $(s)$.

Section~\ref{sec:weighted spacesofsmoothfunctions} is written with the spaces $\asmooth$ in Theorem~\ref{the:mainresults} in mind, but the actual main result of this section, Theorem~\ref{the:generalFrechet}, is in arbitrary dimension and considerably more general. Its proof is based on one of the Sobolev Embedding Theorems.

In Section~\ref{sec:regularityandtopology} the Markov inequality in Theorem~\ref{the:aMarkov} is combined with the results from Section~\ref{sec:weighted spacesofsmoothfunctions}. It is shown that, if $f\in\araspace$, its Fourier series does not just converge in $\araspace$, but in fact in $\asmooth$. This gives the continuous non-trivial inclusion in Theorem~\ref{the:mainresults}.

Section~\ref{sec:alphaistwo} is concerned with the case where $\alpha=2$, where it is possible to describe $\tworaspace$ explicitly.

Section~\ref{sec:questions} contains some possibilities as a basis for further research. We also include a result indicating, even a bit stronger than already in the rest of the paper, how for the problem at hand classical results in approximation theory are naturally intertwined with functional analytic methods.

\begin{remark}
The papers \cite{Zerner} and \cite{Zeriahi} give a theoretical foundation to the principle in the theory of special functions that ``smoothness gives good convergence". Indeed, if, in the context of those papers, $f\in L_2(\Omega,\mu)_{\textup{ra}}$ (i.e., if $f$ is sufficiently regular), then the Fourier series of $f$ does not just converge in $L_2(\Omega,\mu)$, but actually in the topology of $L_2(\Omega,\mu)_{\textup{ra}}$. Typically this will imply that one can partially differentiate the series termwise an arbitrary number of times, and the resulting series will then converge uniformly to the corresponding derivative of $f$. It may be that these very general results are presently not as well known among researchers in the theory of special functions as they deserve.
\end{remark}

\section{Rapid approximation in separable Hilbert spaces}\label{sec:rapidapproximation}

In this section we define the subspace of a Hilbert space that consists of elements that can be approximated rapidly by elements lying in increasing subspaces that are defined naturally in terms of a fixed orthonormal basis. This subspace is supplied with a natural topology in which it is a Fr\'echet space, and it is shown to be topologically isomorphic to $(s)$, cf.\ Theorem~\ref{the:HilbertFrechet}.
We start with a preparatory result in which subspaces of a given Fr\'echet space are supplied with a new Fr\'echet topology that is stronger than the induced topology.

\begin{proposition}\label{prop:generalFrechet}
Let $X$\ be a Fr\'echet space, with topology induced by a finite or countably infinite separating set $\{p_n : n\in P\}$ of seminorms on $X$. Suppose that  $I$ is a finite or countably infinite index set, that, for each $i\in I$, $\{q_{i,j} : j\in J_i\}$ is a set of continuous seminorms on $X$ of arbitrary cardinality, and that $w_i: J_i\to\mathbb R_{\geq 0}$\ is a nonnegative weight on $J_i$. Let
\begin{equation*}
X_I=\{x\in X : \sup_{j\in J_i} w_i(j)q_{i,j}(x)<\infty\textup{ for all }i\in I\}.
\end{equation*}
Then $X_I$ is a linear subspace of $X$, and, for each $i\in I$, the map $q_i: X_I\to\R_{\geq 0}$, defined by
\begin{equation*}
q_i (x)=\sup_{j\in J_i} w_i(j)q_{i,j}(x)\quad(x\in X_I),
\end{equation*}
is a seminorm on $X_I$. Furthermore, $X_I$\ is a Fr\'echet space when supplied with the locally convex topology induced by the separating family $\{p_n : n\in P\}\cup \{q_i : i\in I\}$ of seminorms on $X_I$, this topology is independent of the choice of the family $\{p_n : n\in P\}$ inducing the original topology on $X$, and the inclusion map $X_I\subset X$ is continuous.
\end{proposition}

\begin{proof}
The routine verifications -- for which the continuity of the $q_{i,j}$ is not needed -- that $X_I$ is a linear subspace of $X$, that the $q_i$ $(i\in I)$ are seminorms on $X_I$, that the topology on $X_I$ does not depend on the choice of the family $\{p_n : n\in P\}$, and that the inclusion map is continuous are left to the reader.

To see that $X_I$ is Fr\'echet, we first note that it is metrizable, since the separating family $\{p_n : n\in P\}\cup \{q_i : i\in I\}$ is at most countably infinite.

As to the completeness of $X_I$, let $\seq{x_k}\subset X_I$ be a Cauchy sequence. Since the $p_n$ $(n\in P)$ are included in the family of seminorms defining the topology on $X_I$, $\seq{x_k}$ is also a Cauchy sequence in the complete space $X$. Let $x$ denote its limit. We must prove that $x\in X_I$ and that $x_k\to x$ in $X_I$.

To show that $x\in X_I$ note that, for each fixed $i\in I$, there exists $C_i\geq 0$ such that $q_i(x_k)\leq C_i$ for all $k$.  That is, $w_i(j)q_{i,j}(x_k)\leq C_i$, for all $j\in J_i$ and all $k$. Since the $q_{i,j}$ are continuous on $X$, this implies that $w_i(j)q_{i,j}(x)\leq C_i$, for all $j\in J_i$. We conclude that $x\in X_I$.

It remains to show that $x_k\to x$ in $X_I$, i.e., that $p_n(x-x_k)\to 0$ $(n\in P)$ and $q_i(x-x_k)\to 0$ $(i\in I)$. The first statement is simply the convergence of $x_k$ to $x$ in $X$, so we turn to the second. Fix $i\in I$ and let $\epsilon >0$. Then there exists $N\in\mathbb N$ such that $q_i(x_k-x_l)<\epsilon/2$, for all $k,l\geq N$, hence $w_i(j)q_{i,j}(x_k-x_l)<\epsilon/2$ for all $k,l\geq N$\ and all $j\in J_i$.  Since the $q_{i,j}$ are continuous on $X$, this implies that $w_i(j)q_{i,j}(x-x_k)\leq \epsilon/2$ for all $k\geq N$ and all $j\in J_i$. We conclude that $q_i(x-x_k)\leq\epsilon/2 <\epsilon$ for all $k\geq N$. Hence $q_i(x-x_k)\to 0$, as required.
\end{proof}

We can now define the subspace of rapidly approximable elements of a Hilbert space and show that in its natural topology it is topologically isomorphic to $(s)$. The choice to start the indexing of the orthonormal basis at 0 is made with the constant polynomials in mind.
\begin{theorem}\label{the:HilbertFrechet}
Let $H$ be a separable Hilbert space with orthonormal basis $\{e_n : n=0,1,2,\ldots\}$. For $n=0,1,2,\ldots$, let $L_n=\textup{Span}\{e_k :\ 0\leq k\leq n\}$. Put $L_{-1}=\{0\}$. Let
\begin{equation*}
\raspace=\{x\in H : \sup_{n\geq 0} n^kd(x,L_{n-1})<\infty \textup{ for all }k=0,1,2,\ldots\},
\end{equation*}
where $0^0$\ is to be read as $1$, and $d(x,L_{n-1})$\ is the distance from $x$ to the closed subspace $L_{n-1}$\ of $H$. For $k=0,1,2,\ldots$, define $q_k: \raspace\to\R_{\geq 0}$ by
\begin{equation*}
q_k(x)=\sup_{n\geq 0} n^kd(x,L_{n-1})\quad(x\in \raspace).
\end{equation*}
Then $\{q_k :\ k=0,1,2,\ldots\}$ is a separating family of seminorms on $\raspace$ that induces a Fr\'echet topology on $\raspace$. Moreover, $x=\sum_{n=0}^\infty a_n(x)e_n\in H$ is in $\raspace$ precisely when $\seq{a_n(x)}\in (s)$, and the map $\iso$ sending $x$ to $\seq{a_n(x)}$ is a topological isomorphism between the Fr\'echet spaces $\raspace$ and $(s)$.
\end{theorem}

\begin{proof}
The topology on $H$ is defined by the norm $||\,.\,||$ and, for $n=-1,0,1,2,\ldots$, the map sending $x\to d(x,L_{n-1})$ is a continuous seminorm on $H$. Therefore Proposition~\ref{prop:generalFrechet} shows that $\raspace$ is a Fr\'echet space in the topology induced by the family $\{||\,.\,||\}\cup\{q_k : k=0,1,2,\ldots\}$ of seminorms on $\raspace$. However, since $||\,.\,||\leq q_0$, the family $\{q_k : k=0,1,2,\ldots\}$ induces the same topology.

It remains to establish the topological isomorphism between $\raspace$ and $(s)$.

First of all, if $x=\sum_{n=0}^\infty a_n(x)e_n\in H$ is in $\raspace$, and $k\in\mathbb N_0$\ is fixed, then there exists $C\geq 0$ such that $n^k\left(\sum_{i=n}^\infty |a_i(x)|^2\right)^{1/2} \leq C$ for all $n$. Then certainly $n^k|a_n(x)|\leq C$ for all $n$, showing that $\seq{a_n(x)}\in (s)$. Conversely, assume $x=\sum_{n=0}^\infty a_n(x)e_n\in H$ with $\seq{a_n(x)}\in (s)$. Fix $k\geq 0$. Then there exists a constant $C\geq 0$ such that $|a_n(x)|\leq C/n^{k+1}$ for all $n\geq 1$. Hence we have, for $n\geq 2$,
\begin{align*}
d(x,L_{n-1})&=\left(\sum_{i=n}^\infty |a_i(x)|^2\right)^{1/2}\\
&\leq \left(\sum_{i=n}^\infty \frac{C^2}{n^{2k+2}}\right)^{1/2}\\
&\leq \left(\int_{n-1}^\infty \frac{C^2}{x^{2k+2}}\,dx\right)^{1/2}\\
&=\frac{C}{\sqrt{2k+1}}\frac{1}{(n-1)^{k+1/2}}.
\end{align*}
This implies that $\sup_{n\geq 0} n^kd(x,L_{n-1})<\infty$, as required.

Thus $\iso :\raspace\to (s)$ is an isomorphism of abstract vector spaces. To see that it is also topological, one could resort to elementary means as above, but it also comes almost for free as a consequence of the completeness of the spaces.
To start with, the continuity of $\iso$ follows from the Closed Graph Theorem
for F-spaces \cite[Theorem~2.15]{Ru}. Indeed, suppose that $x_k\rightarrow
x$ in $\raspace$ and $\iso(x_k)\rightarrow \seq{a_n}$ in $(s)$. Then certainly $|| x_k-x||\leq q_0(x-x_k)\rightarrow 0$, hence the continuity of the Fourier coefficients
on $H$ implies that $\iso(x_k)$ converges to $\iso(x)$ in each coordinate.
But $\iso(x_k)$ also converges to $\seq{a_n}$ in each coordinate. Hence
$\iso(x)=\seq{a_n}$, and the Closed Graph Theorem then shows that $\iso$
is continuous. Since we already know that $\iso$ is an isomorphism of abstract
vector spaces, the Open Mapping Theorem for F-spaces \cite[Theorem~2.11]{Ru}
then implies that $\iso$ is a topological isomorphism.
\end{proof}

\section{Weighted spaces of smooth functions}\label{sec:weighted spacesofsmoothfunctions}

It is not obvious that the elements of the spaces $\araspace$ are actually smooth if $\alpha>1$, but this will follow from the Markov inequality \eqref{eq:Markov} and the completeness of the spaces in the following rather general result. The statement on convergence in the spaces will allow us to improve our understanding of the topology on $\araspace$.

In this section, we employ the usual notation for the differential operator $D^\alpha$ of order $|\alpha|$ corresponding to a multi-index $\alpha\in\N_0^d$.

\begin{theorem}\label{the:generalFrechet}
Let $\Omega$ be a nonempty open subset of $\R^d$ and $w:\Omega\to (0,\infty)$ a strictly positive measurable function on $\Omega$. For  $1\leq p<\infty$, define
\begin{equation*}
\wpsmooth =\{f\in C^{\infty}(\Omega):D^\alpha f\in L_{p}(\Omega, w\,dx) \textup{ for all } \alpha\in\N_0^d\},
\end{equation*}
and supply $\wpsmooth$ with the locally convex topology induced by the family of seminorms $\{\wpaseminorm : \alpha\in\N_0^d\}$, defined by
\begin{equation*}
\wpaseminorm(f)=\left\{\int_{\Omega}|D^\alpha f(x)|^{p}w(x)dx\right\}^{\frac{1}{p}}\quad(f\in\wpsmooth).
\end{equation*}

Suppose that there exists an open cover $\Omega=\bigcup_{i \in I}\Omega_{i}$, where $I$ is an arbitrary index set, such that, for each $i \in I$, $\Omega_i$, there exists $C_i>0$ such that $w(x)\geq C_i$ for all $x\in \Omega_i$. Then $\wpsmooth$ is a Fr\'echet space.

If $f_n\to f$ in $\wpsmooth$ and $\alpha\in\N_0^d$, then $D^\alpha f_n\to D^\alpha f$ uniformly on $\Omega_i$ for all $i\in I$, and consequently also on all compact subsets of $\Omega$.
\end{theorem}

Before we turn to the proof, let us note that such a cover always exists if $w$ is continuous and strictly positive: let $I=\Omega$ and take a small open ball around each $x\in\Omega$ that is contained in $\Omega$. Hence $\wpsmooth$ is a Fr\'echet space for such $w$, for all $1 \leq p<\infty$, and the convergence statement is valid.

As to the proof of Theorem~\ref{the:generalFrechet}, we first note that $\{\wpaseminorm :\alpha\in\N_0^d\}$ is trivially a separating family of seminorms on $\wpsmooth$. Indeed, since $w$ is strictly positive, $q_{w,p,0}$ is in fact a norm. Since the family is countable, \cite[Chapter~IV~Proposition~2.1]{Conway} shows that $\wpsmooth$ is metrizable. Hence it remains to show that $\wpsmooth$ is complete, and that the statement on convergence holds.

Our proof for this is based on a part of the Sobolev Imbedding Theorem as formulated in \cite{Adams}. We recall the relevant notions for the convenience of the reader. Let $m=0,1,2,\ldots$ be an integer and let $1\leq p\leq\infty$. If $\Omega$ is a nonempty possibly unbounded open subset of $\R^d$, the Sobolev space $W^{m,p}(\Omega)$ is defined as

\begin{equation*}
W^{m,p}(\Omega) = \{f\in L_{p}(\Omega, dx):D^{\alpha}f\in L_{p}(\Omega, dx)\;\textup{for}\;0\leq|\alpha|\leq m\},
\end{equation*}
with the usual identification of functions agreeing almost everywhere. Here $D^\alpha f$ is the weak (distributional) derivative of $f$. As is well known (see \cite[3.3]{Adams} for a proof), $W^{m,p}(\Omega)$ is a Banach space when equipped with the norm defined, for $f\in W^{m,p}(\Omega)$, by
\begin{equation*}
\|f\|_{m,p}=
\begin{cases}
\left(\sum_{0\leq|\alpha|\leq m}\|D^\alpha f\Vert_p^p\right)^{1/p} &\textup{if }1\leq p<\infty;\\
\max_{0\leq|\alpha|\leq m} \Vert D^\alpha f\Vert_\infty&\textup{if }p=\infty.
\end{cases}
\end{equation*}
The part of the Sobolev Imbedding Theorem we will need embeds these spaces $W^{m,p}(\Omega)$ continuously into spaces of functions with a certain minimal degree of regularity. If $\Omega$ is a nonempty open subset of $\R^d$, and $j=0,1,2,\ldots$, then $C_B^j(\Omega)$ is defined by
\begin{equation*}
C_{B}^{j}(\Omega)=\{u\in C^{j}(\Omega):D^{\alpha}u\;\textup{is\;bounded\;on}\;\Omega\;\textup{for all }\;0\leq|\alpha|\leq j\}.
\end{equation*}
Then (see \cite[1.27]{Adams}) $C_B^j(\Omega)$\ is a Banach space when supplied with the norm defined, for $f\in C_B^j(\Omega)$, by
\begin{equation*}
\| f\|_j=\max_{0\leq |\alpha|\leq j}\| D^\alpha f\|_\infty.
\end{equation*}
We also recall that $\Omega$ is said to satisfy the cone condition if there exists a finite cone $C$ such that each $x\in\Omega$ is the vertex of a finite cone $C_x$ that is contained in $\Omega$ and that is obtained from $C$ by a rigid motion.

The case of the Sobolev Imbedding Theorem we will use is then as follows \cite[Theorem~4.12.I.A]{Adams}.

\begin{theorem}\label{the:Sobolev}
Let $\Omega$ be a nonempty open subset of $\R^d$ satisfying the cone condition. Let $j\geq 0$ and $m\geq 1$ be integers and let $1\leq p<\infty$. If either $mp>n$ or $m=n$ and $p=1$, then
\begin{equation*}
W^{j+m,p}(\Omega)\subset C_{B}^{j}(\Omega),
\end{equation*}
and the inclusion map is continuous.
\end{theorem}

We can now finish the proof of Theorem~\ref{the:generalFrechet}.

\begin{proof}[Conclusion of the proof of Theorem~\ref{the:generalFrechet}]
We start by establishing the following claim: If $\seq{f_n}$ is a Cauchy sequence in $\wpsmooth$, then there exists a function $f\in C^{\infty}(\Omega)$ such that $D^\alpha f_n (x)\to D^\alpha f(x)$ uniformly on $\Omega_i$, for all $i\in I$ and all $\alpha\in\N_0^d$. To see this, first note that we may as well assume that all $\Omega_i$ in Theorem~\ref{the:generalFrechet} satisfy the cone condition. Indeed, for each $x\in\Omega$ one can choose an open Euclidean ball $B_x$ that is contained in one of the $\Omega_i$, and then $\Omega=\bigcup_{x\in\Omega}B_x$ is an open cover with all required properties. If $i\in I$, then, since $w>C_i$ on $\Omega_i$, the map $\textup{Res}_i$ given by restricting functions on $\Omega$ to $\Omega_i $ gives natural continuous maps $\textup{Res}_{i,m}: \wpsmooth\to W^{m,p}(\Omega_i)$, for all $m=0,1,2,\dots$. Since $\Omega_i$ satisfies the cone condition, Theorem~\ref{the:Sobolev} then implies that, by composition of continuous inclusions, restriction gives continuous maps $\textup{Res}_{i,j}:\wpsmooth\to C_B^{j}(\Omega_i)$, for all $j=0,1,2,\ldots$.
Thus, if $\seq{f_n}$ is a Cauchy sequence in $\wpsmooth$, then $\seq{\textup{Res}_{i,j}f_n}$ is a Cauchy sequence in $C_B^j(\Omega_i)$, for all $i\in I$ and all $j=0,1,2,\ldots$. Since these spaces are Banach spaces, it is then not difficult to see that, for each $i\in I$, there exists $\phi_i\in C^{\infty}(\Omega_{i})$ such that, for every $\alpha\in\N_0^d$, $\seq{D^\alpha \textup{Res}_i f_n}$ converges uniformly to $D^\alpha \phi_i$ on $\Omega_i$.
Since the $\phi_i$ must then clearly agree on intersections of the $\Omega_i$, they patch together to yield $f\in C^{\infty}(\Omega)$ as requested. This concludes the proof of the claim.

Now let $\seq{f_n}$ be a Cauchy sequence in $\wpsmooth$. As a consequence of the first part of the proof, there exists $f\in C^{\infty}(\Omega)$ such that $D^\alpha f_n (x)\to D^\alpha f(x)$ for all $x\in \Omega$ and all $\alpha\in\N_0^d$. Actually, $f$ is in $\wpsmooth$ and $\seq{f_n}\to f$ in $\wpsmooth$, so that $\wpsmooth$ is complete. To see this, note that $\seq{D^\alpha f_n}$ is a Cauchy sequence in $L_{p}(\Omega,w\,dx)$, for all $\alpha\in\N_0^d$. Hence for each $\alpha\in\N_0^d$ there exists $g_\alpha\in L_p(\Omega, w\,dx)$ such that $\seq{D^\alpha f_n}\to g_\alpha$, and there exists a subsequence $\seq{D^\alpha f_{n_k}}$ such that $\seq{D^\alpha f_{n_k}(x)}$ converges $w(x)\,dx$-almost everywhere (and hence Lebesgue almost everywhere, since $w$ is strictly positive) to $g_\alpha (x)$ as $k$ tends to infinity. But $\seq{D^\alpha f_{n_k}(x)}$ also converges to $D^\alpha f(x)$ for all $x\in \Omega$. Hence $g_\alpha=D^\alpha f$ almost everywhere, for all $\alpha\in\N_0^d$. We conclude that $f\in\wpsmooth$ and that $\seq{f_n}\to f$ in $\wpsmooth$. This concludes the proof of the completeness of $\wpsmooth$.

The convergence statement is already implicit in the first part of the proof. Indeed, as observed in that first part, restriction gives continuous maps  $\textup{Res}_{i,j}:\wpsmooth\to C_B^{j}(\Omega_i)$, for all $j=0,1,2,\ldots$. Applying these to a convergent sequence $f_n\to f$ gives the statement on uniform convergence of all $\seq{D^\alpha f_n}$ on all $\Omega_i$. The uniform convergence on all compact subsets of $\Omega$ is then also clear.
\end{proof}

\begin{remark}
We could also have introduced our space $\wpsmooth$ as
\begin{equation*}
W_{w,p}^\infty =\{f:\Omega\to\mathbb C \textup{ is measurable and } D^\alpha f\in L_{p}(\Omega, w\,dx) \textup{ for all } \alpha\in\N_0^d\},
\end{equation*}
with the usual identification of functions agreeing almost everywhere, and where $D^\alpha f$ is now the weak derivative of $f$. Indeed, an argument as in the above proof, combining local restrictions with the regularity statement in Theorem~\ref{the:Sobolev}, shows that elements of $W_{w,p}^\infty$ are necessarily smooth. Hence $\wpsmooth=W_{w,p}^\infty$ and, if one prefers, one can think of $\wpsmooth$ as a weighted Sobolev space of infinite order.
\end{remark}

\section{$\araspace$: regularity and topology}\label{sec:regularityandtopology}

We can now establish the regularity of elements of $\araspace$ and get a better grip on the topology of this space. Ultimately this is all based on the Markov inequality \eqref{eq:Markov}, that is used in the proof of the key Proposition~\ref{prop:keyprop}, and the Sobolev Embedding Theorem~\ref{the:Sobolev}, that is used in the proof of Theorem~\ref{the:generalFrechet}.

As a first preparatory result, we note the following special case of Theorem~\ref{the:generalFrechet}.

\begin{theorem}\label{the:aFrechet}
For $\alpha>0$, let
\begin{equation*}
\asmooth =\{f\in C^{\infty}(\R):f^{(j)}\in \altwo ,\, j=0,1,2,\cdots\}.
\end{equation*}
Supply $\asmooth$ with the locally convex topology induced by the family of seminorms $\{\ajseminorm : j=0,1,2,\ldots\}$, defined, for $j=0,1,2,\cdots$,  by
\begin{equation}\label{eq:ajseminormdef}
\ajseminorm(f)=\|f^{(j)}\|_\anorm\quad(f\in\asmooth).
\end{equation}
 Then $\asmooth$ is a Fr\'echet space. If $f_n\to f$ in $\araspace$, then $f_n^{(j)}\to f^{(j)}$ uniformly on  compact subsets of $\R$, for all $j\geq 0$.
 \end{theorem}

The second preparatory result, based on the Markov inequality \eqref{eq:Markov}, is as follows.

\begin{proposition}\label{prop:keyprop}
Let $\alpha>1$\ and let $\afourier^{-1}:(s)\to\altwo$ be the map assigning to $\seq{a_n}\in (s)$ the corresponding element $\sum_{n=0}^\infty a_n P_{\alpha,n}$ of $\altwo$. Then the series actually converges in $\asmooth$, and consequently
$\sum_{n=0}^\infty a_n P_{\alpha,n}^{(j)}$ converges uniformly on compact
subsets of $\R$ to $\left(\sum_{n=0}^\infty a_n P_{\alpha,n}\right)^{(j)}$, for all $j=0,1,2,\ldots$.

Furthermore, the map $\afourier^{-1}:(s)\to\asmooth$ thus obtained is continuous.
\end{proposition}

\begin{proof}
We start by showing that the series  $\sum_{n=0}^\infty a_n P_{\alpha,n}$ converges in $\asmooth$ for all $\seq{a_n}\in(s)$.  Since we know from Theorem~\ref{the:aFrechet} that $\asmooth$ is complete, this will follow once we know that $\left(\sum_{n=0}^{N} a_n P_{\alpha,n}\right)_{N=0}^\infty$ is a Cauchy sequence in $\asmooth$. For this it is clearly sufficient to show that $\sum_{n=0}^\infty \ajseminorm(a_n P_{\alpha,n})<\infty$ for $j=0,1,2,\ldots$, where the $\ajseminorm$ are the seminorms in \eqref{eq:ajseminormdef}. As to this, we note that it follows by iterating the Markov inequality \eqref{eq:Markov} for $p=2$ that, for $n,j=0,1,2,\ldots$,
\begin{align*}
\|P^{(j)}_{\alpha,n}\|_\anorm & \leq C_{\alpha,2}^{j}\left\{n(n-1)\cdots(n-j+1)\right\}^{1-\frac{1}{\alpha}}\|P_{\alpha,n}\|_\anorm\notag
\\ &\leq C_{\alpha,2}^{j}n^{j\left(1-\frac{1}{\alpha}\right)}\|P_{\alpha,n}\|_\anorm\notag
\\ &= C_{\alpha,2}^{j}n^{j\left(1-\frac{1}{\alpha}\right)}.
\end{align*}
Thus, for $j=0,1,2,\ldots$,
\begin{equation}\label{eq:finite}
\sum_{n=0}^\infty \ajseminorm(a_n P_{\alpha, n})=\sum_{n=0}^\infty |a_n|\|P_{\alpha,n}^{(j)}\|_\anorm \leq C_{\alpha,2}^j\sum_{n=0}^\infty |a_n|n^{j\left(1-\frac{1}{\alpha}\right)}.
\end{equation}
Since $\seq{a_n}\in(s)$, the right hand side in \eqref{eq:finite} is finite, as required, and this concludes the proof of the claim. The statement on uniform convergence then follows from Theorem~\ref{the:aFrechet}.

For the continuity of $\afourier^{-1}:(s)\to\asmooth$, fix $j\geq 0$. Choose an integer $k\geq j\left(1-\frac{1}{\alpha}\right)$. Then for arbitrary $\seq{a_n}\in (s)$ we have, using \eqref{eq:finite},
\begin{align*}
\ajseminorm\left(\afourier^{-1}(\seq{a_n}) \right)&=\ajseminorm\left(\sum_{n=0}^\infty a_n P_{\alpha,n}\right)\\
&\leq \sum_{n=0}^\infty\ajseminorm (a_n P_{\alpha,n})\\
&\leq C_{\alpha,2}^j\sum_{n=0}^\infty |a_n|n^{j\left(1-\frac{1}{\alpha}\right)}\\
&\leq C_{\alpha,2}^j\sum_{n=0}^\infty |a_n|n^k\\
&\leq C_{\alpha,2}^j |a_0| + C_{\alpha,2}\sum_{n=1}^\infty |a_n| n^{(k+2)-2}\\
&\leq C_{\alpha,2}^j\sup_n |a_n| + C_{\alpha,2}^j\left(\sum_{n=1}^\infty \frac{1}{n^2}\right)\sup_n n^{k+2}|a_n |.
\end{align*}
Since $\seq{a_n}\to\sup_n|a_n|$ and $\seq{a_n}\to\sup_n n^{k+2}|a_n|$ are elements of the family of seminorms defining the topology on $(s)$, we conclude (see, e.g., \cite[Proposition~1.2.8]{KaRi}) that $\afourier^{-1}$ is continuous.
\end{proof}

It is now a simple matter to combine this. Since $\aiso:\araspace\to (s)$ is continuous by Theorem~\ref{the:HilbertFrechet}, and $\aiso^{-1}:(s)\to \asmooth$ is continuous by Proposition~\ref{prop:keyprop}, we see that $\aiso^{-1}\circ\aiso: \araspace\to\asmooth$ is continuous. But this composition is the identity on $\araspace$. Hence we have the following.

\begin{theorem}\label{the:regularityandtopology}
Let $\alpha>1$. Then $\araspace\subset\asmooth$, and the inclusion map is continuous. If $f_n\to f$ in $\araspace$, then $f_n^{(j)}\to f^{(j)}$ uniformly on  compact subsets of $\R$, for all $j\geq 0$.
\end{theorem}

\section{$\twoltwo$: topological isomorphism with the Schwartz space}\label{sec:alphaistwo}

Although for general $\alpha>1$ Theorem~\ref{the:regularityandtopology} establishes some non-trivial basic facts for $\araspace$, a more concrete description of this space would be desirable, even if only as a set. At present this seems out of reach, but there is an exception if $\alpha=2$. In that case, $\tworaspace$ is topologically isomorphic with the Schwartz space $\S(\R)$ of rapidly decreasing functions via a multiplication map, cf.\ Theorem~\ref{the:alphaistwo}. The idea is to combine the topological isomorphism of $\tworaspace$ with $(s)$ and a known topological isomorphism between $(s)$ and the $\S(\R)$.

The latter topological isomorphism between $(s)$ and $\S(\R)$ involves Hermite functions, whose definitions we now recall.
For $n\geq 0$ and $x\in\R$, let, as in \cite[p.~142]{ReeSim}, or \cite[(1.1.2) and (1.1.18)]{Tha},
\begin{equation}\label{eq:hermitefunctionsdef}
h_n(x)=(-1)^n\left( 2^n n! \sqrt{\pi} \right)^{-\frac{1}{2}}e^{\frac{1}{2}x^2}\frac{d^n}{dx^n}e^{-x^2}
\end{equation}
Then the Hermite functions $\seq{h_n}$ form an orthonormal basis of $L^2(\R,dx)$. This is stated as \cite[Lemma~V.3]{ReeSim} with a reference to the exercises for the proof. Alternatively, \cite[22.2.14]{AbrSte} or \cite[(5.5.1)]{Sze} gives orthonormality, and \cite[(5.7.2)]{Sze} gives completeness; the latter also follows from \cite[Corollary~6.34]{Jeu}.

If $f\in  \ltwo$, then its Fourier coefficients with respect to this orthonormal basis are given, for $n\geq 0$, by
\begin{align*}\label{eq:transformfunctionsHermite}
\fourierhermite(f)_n&=\int_\R f(x)h_n(x)\,dx.
\end{align*}
If $f\in\S(\R)$, then the sequence $\seq{\fourierhermite(f)_n}$ is not just in $\ell_2$, but in fact in $(s)$. Actually, by \cite[Theorem~V.13]{ReeSim} and its proof, see also \cite[p.~262]{Schw}, $\fourierhermite:\S(\R)\to (s)$ is an isomorphism of topological vector spaces between $\S(\R)$ and $(s)$, and if $f\in\S(\R)$, then the series $\sum_{n=0}^\infty \fourierhermite(f)_nh_n$ converges to $f$ in the topology of $\S(\R)$.

We will now relate this to our setup, as follows.
If we define
$$
\psi(x)=e^{-x^2/2}\quad(x\in\R),
$$
then it is clear from \eqref{eq:hermitefunctionsdef} that
\begin{equation}\label{eq:hermitepolypsi}
h_n=Q_n\psi \quad(n\geq 0)
\end{equation} for some polynomial $Q_n$ of degree $n$. Hence $\seq{Q_n}$ is the system of orthonormal polynomials for the weight $\psi^2$
on $\R$. They are essentially the $P_{2,n}$, up to a constant and a dilation. Establishing notation to make this precise, if $f:\R\to\C$ is a function, and $r>0$, we let
\begin{equation*}
\delta_r f(x)=f(rx) \quad (x\in\R)
\end{equation*}
be the corresponding dilation of $f$. Now, for $n,m\geq 0$, a change of the variable of integration in the second step gives
\begin{align*}
\delta_{n,m}&=\int_\R P_{2,n}(x)P_{2,m}(x)e^{-2x^2}\,dx\\
&=\frac{1}{\sqrt 2}\int_\R P_{2,n}(s/\sqrt 2)P_{2,n}(2/\sqrt 2)e^{-s^2}\,ds\\
&=\int_\R \left(\frac{\delta_{1/\sqrt 2}P_{2,n}}{2^{1/4}}\right)(s) \left(\frac{\delta_{1/\sqrt 2}P_{2,m}}{2^{1/4}}\right)(s)\psi^2(s)\,ds.
\end{align*}
Since dilation preserves the degree of a polynomial, we conclude that
\begin{equation}\label{eq:polyrelation}
Q_n=\frac{\delta_{1/\sqrt 2}P_{2,n}}{2^{1/4}}\quad(n\geq 0).
\end{equation}
If $f\in\twoltwo$, then $(\delta_{1/\sqrt 2}f)\psi\in \ltwo$. For such $f$ we compute, for $n\geq 0$, using \eqref{eq:polyrelation} and \eqref{eq:hermitepolypsi},
\begin{align}\label{eq:relatingcoefficients}
(f,P_{2,n})_{\twoltwo}&=\int_\R f(x)P_{2,n}e^{-2x^2}\,dx\notag \\
&=\frac{1}{\sqrt 2}\int_\R f(s/\sqrt 2) P_{2,n}(s/\sqrt 2)e^{-s^2}\,ds\notag \\
&=\int_\R \left(\frac{\delta_{1/\sqrt 2}f}{2^{1/4}}\psi\right)(s)\left(\frac{\delta_{1/\sqrt 2}P_{2,n}}{2^{1/4}}\psi \right)(s)\,ds\notag \\
&=\int_\R \left(\frac{\delta_{1/\sqrt 2}f}{2^{1/4}}\psi\right)(s)\left(Q_n\psi \right)(s)\,ds\notag \\
&=\fourierhermite\left(\frac{\delta_{1/\sqrt 2}f}{2^{1/4}}\psi\right)_n.
\end{align}
When combining the topological isomorphism $\aiso:\tworaspace\to(s)$ and $\fourierhermite^{-1}:(s)\to \S(\R)$, we obtain a topological isomorphism $\aiso\circ\fourierhermite^{-1}: \araspace\to\S(\R)$. Concretely, if $f\in\tworaspace$, then, using \eqref{eq:relatingcoefficients},
\begin{align*}
 \aiso\circ\fourierhermite^{-1}(f)&=\sum_{n=0}^\infty (f,P_{2,n})_{\twoltwo} h_n\notag\\
 &=\sum_{n=0}^\infty\fourierhermite\left(\frac{\delta_{1/\sqrt 2}f}{2^{1/4}}\psi\right)_n h_n\\
 &=\frac{\delta_{1/\sqrt 2}f}{2^{1/4}}\psi.
\end{align*}
We conclude that the map
$$
f\mapsto \frac{\delta_{1/\sqrt 2}f}{2^{1/4}}\psi
$$
is a topological isomorphism between $\tworaspace$ and $\S(\R)$. Since dilations are automorphisms of $\S(\R)$, the same is then true for $f\mapsto\delta_{\sqrt 2}((\delta_{1/\sqrt 2}f) \psi)=f\psi^2$. All in all, we have the following description of $\tworaspace$.

\begin{theorem}\label{the:alphaistwo}
The elements $f$ of $\tworaspace$ are precisely all functions of the form
\begin{equation}\label{eq:isomorphism}
f(x)=g(x)e^{x^2}\quad(x\in\R),
\end{equation}
where $g\in\S(\R)$. That is, if $f\in \twoltwo$, then the following are equivalent:
\begin{enumerate}
\item $\sup_n n^k d(f,\Pi_{n-1})_{\twoltwo}<\infty$, for all $k=0,1,2,\ldots$.
\item There exists a Schwartz function $g\in\S(\R)$ as in \eqref{eq:isomorphism}.
\end{enumerate}

Moreover, the bijection in \eqref{eq:isomorphism} between $\tworaspace$ and $\S(\R)$ is a topological isomorphism.
\end{theorem}

\section{Possibilities}\label{sec:questions}

Although Theorem~\ref{the:alphaistwo} gives a complete answer for $\alpha=2$, the results for general $\alpha>1,\,\alpha\neq 2$ are still not complete, with Theorem~\ref{the:mainresults} and Remark~\ref{rem:ditzianpaper} containing what appears to be known at this moment.

It is tempting to try to extrapolate Theorem~\ref{the:alphaistwo}, and suggest the possibility that $\araspace$ consists of (or at least contains) the functions that are equal to $g W_\alpha^{-1}$ on $(R,\infty)$ for some $g\in\S(\R)$ and $R>0$, and likewise at $-\infty$. Certainly \eqref{eq:bounded} shows that there exist such $g$ with $g$ bounded, but this is still very far from $g$ being rapidly decreasing, and we refrain from stating a conjecture.

Likewise, at the moment we have no evidence, also not for $\alpha=2$, that $\asmooth$ may or not be equal to $\araspace$. Still it is interesting to investigate what would follow if this were actually the case. This is done in our final result. It shows that the validity of a stronger form of the Jackson-type inequality \eqref{eq:aJackson} would not only imply that $\asmooth=\araspace$ as sets, but is actually equivalent with the equality of these sets.

\begin{proposition}\label{prop:equivalences}
Let $\alpha>1$. Then the following are equivalent:
\begin{enumerate}
 \item $\asmooth=\araspace$ as sets;
 \item $\asmooth=\araspace$ as topological vector spaces;
 \item For each $f\in \asmooth$ and each $k=0,1,2,\ldots$, there exists a constant $C$ such that, for all $n=0,1,2,\ldots$,
 \begin{equation}\label{eq:aactualJackonpointwise}
 n^k d(f,\Pi_n)_{\altwo}\leq C.
 \end{equation}
 \item For each $k=0,1,2,\ldots$, there exist a constant $C$, an integer $r>0$, and integers $0\leq j_1<j_2<\cdots<j_r$, with the property that, for all $f\in\asmooth$ and all $n=0,1,2,\ldots$,
 \begin{equation}\label{eq:aactualJacksonuniform}
 n^k d(f,\Pi_n)_{\altwo}\leq C \sum_{i=1}^r\| f^{(j_i)}\|_{\altwo}.
\end{equation}
\end{enumerate}
\end{proposition}

\begin{proof}
Clearly (2) implies (1). Conversely, if (1) holds, then the Open Mapping Theorem shows that the continuous inclusion map $\araspace\subset\asmooth$ is actually a topological isomorphism. Hence (1) implies (2). It is clear that (4) implies (3). Since (3) is equivalent to stating that $\asmooth\subset\araspace$, and the reverse inclusion is already known to be true, (3) implies (1). Hence the proof will be complete once we show that (2) implies (4). For this, we fix $k\geq 0$ and note (cf.~Theorem~\ref{the:HilbertFrechet}) that $\tilde q_k:\araspace\to\R_{\geq 0}$, defined by
\begin{equation*}
\tilde q_k(f)=\sup_{n\geq 0} n^k d(f,\Pi_n)_{\altwo}\quad(f\in\araspace),
\end{equation*}
is a continuous seminorm on $\araspace$. Hence, by assumption, it is a continuous seminorm on $\asmooth$, where the topology is defined by the family $\{\ajseminorm : j=0,1,2,\ldots\}$ of seminorms as in \eqref{eq:ajseminormdef}.  But then, by \cite[Proposition~1.2.8]{KaRi}, there exist a constant $C$, an integer $r>0$, and integers $0\leq j_1<j_2<\cdots<j_r$, such that, for all $f\in\asmooth$,
\begin{equation*}
 \tilde q(f)\leq C \sum_{i=1}^r q_{\alpha,j_i}(f).
\end{equation*}
This is the statement in (4).
\end{proof}

Even though Proposition~\ref{prop:equivalences} is only a ``what if''-result, it still brings to the foreground the potential use of combining results in approximation theory of a classical nature with functional analytic methods. This has already been implicit in the rest of the paper, but here it is particularly clear.

For example, it is obvious that (4) implies that $\asmooth\subset\araspace$, but for the converse inclusion (which we know to be true) we used the completeness of $\asmooth$, which was ultimately based on the Sobolev Embedding Theorem.

The fact that the statement in (1) about sets implies a Jackson-type inequality as in \eqref{eq:aactualJacksonuniform} is perhaps even more illustrative. Indeed, once we know that the spaces in (2) are complete, and that one of the inclusions is continuous (which we know to be true by proof using the Markov inequality), the Open Mapping Theorem shows immediately that, if these spaces are equal, they must then be topologically isomorphic. As in the above proof, the Jackson-type inequality \eqref{eq:aactualJacksonuniform} is then a direct consequence of a general functional analytic principle for continous seminorms on locally convex spaces.

The most remarkable consequence of mixing classical approximation theory with functional analysis, however, seems to be the equivalence of (3) and (4). There does not seem to be an a priori reason why a pointwise Jackson-type inequality as in \eqref{eq:aactualJackonpointwise}, with an ``undetermined'' right hand side, should imply a uniform inequality as in \eqref{eq:aactualJacksonuniform}, with a right hand side as occurring in the literature. Nevertheless this must be the case, as a consequence of the Markov inequality, the Sobolev Embedding Theorem and the Open Mapping Theorem combined.

\bibliographystyle{amsplain}

\end{document}